\documentclass[11pt,leqno]{amsart}
\usepackage[dvipsnames]{xcolor}
\usepackage{amsmath}
\usepackage{amsfonts}
\usepackage{amssymb}
\usepackage{graphicx}
\usepackage{hyperref} 
\usepackage{mathrsfs} 
\usepackage{relsize} 

\theoremstyle{plain}
\newtheorem{theorem}{Theorem}

\newtheorem{conjecture}[theorem]{Conjecture}
\newtheorem{corollary}[theorem]{Corollary}
\newtheorem{definition}[theorem]{Definition}
\newtheorem{lemma}[theorem]{Lemma}

\newtheorem{proposition}[theorem]{Proposition}

\theoremstyle{definition}

\newtheorem{remark}[theorem]{Remark}

\numberwithin{equation}{section}
\numberwithin{theorem}{section}

\newcommand{\dive}{\operatorname{div}}
\renewcommand{\div}{\operatorname{div}}

\newcommand{\Hess}{\operatorname{Hess}}

\newcommand{\rr}{\mathbb{R}}

\newcommand{\nn}{\mathbb{N}}

\newcommand{\bb}{\mathbb{B}}

\newcommand{\riem}{\operatorname{Riem}}

\newcommand{\ric}{\operatorname{Ric}}

\newcommand{\supp}{\operatorname{supp}}
\newcommand{\tr}{\operatorname{trace}}

\newcommand{\dv}{\,\mathrm{dv}}
\newcommand{\dx}{\,\mathrm{d}x}

\newcommand{\sgn}{\operatorname{sgn}}

\def\XXint#1#2#3{{\setbox0=\hbox{$#1{#2#3}{\int}$}
     \vcenter{\hbox{$#2#3$}}\kern-.5\wd0}}

\newcommand{\inj}{\mathrm{inj}}

\newcommand{\dist}{\mathrm{dist}}

\renewcommand{\a}{\alpha}

\newcommand{\ve}{\varepsilon}

\renewcommand{\d}{\delta}

\renewcommand{\l}{\lambda}

\newcommand{\vp}{\varphi}

\newcommand{\CL}{\mathcal{L}}

\newcommand{\C}{\mathscr{C}}

\renewcommand{\L}{\mathscr{L}}

\begin{document}


\title[Complete manifolds are $L^{p}$-positivity preserving]{$L^{p}$ Positivity Preserving and a conjecture by M. Braverman, O. Milatovic and M. Shubin}
\author{Stefano Pigola}
\address{Universit\`a degli Studi di Milano-Bicocca\\ Dipartimento di Matematica e Applicazioni \\ Via Cozzi 55, 20126 Milano - ITALY}
\email{stefano.pigola@unimib.it}
\author{Giona Veronelli}
\address{Universit\`a degli Studi di Milano-Bicocca\\ Dipartimento di Matematica e Applicazioni \\ Via Cozzi 55, 20126 Milano - ITALY}
\email{giona.veronelli@unimib.it}
\date{\today}

\begin{abstract}
In this paper we prove that a complete Riemannian manifold is $L^p$-positivity preserving for any $p\in(1,\infty)$. This means that any $L^p$ function which solves $(-\Delta + 1)u\ge 0$ in the sense of distributions is necessarily non-negative. In particular, the case $p=2$ of our result answers in the affermative a conjecture formulated by M. Braverman, O. Milatovic and M. Shubin in 2002. The two main ingredients are a new \textsl{a-priori} regularity result for 
 positive subharmonic distributions, which in turn permits to prove a Liouville type theorem, and a Brezis-Kato inequality on Riemannian manifolds. Both these results rely on a smooth monotonic approximation of distributional solutions of $\Delta u \ge \lambda(x) u$ of independent interest. \end{abstract}

\maketitle

\section{Introduction and main result}

\subsection{Basic notation}

Let $(M,g)$ be a connected, possibly incomplete, $n$-dimensional Riemannian manifold, $n \geq 2$, endowed with its Riemannian measure $\dv$. Unless otherwise specified, integration will be always performed with respect to this measure. The Riemannian metric $g$ gives rise to the intrinsic distance $\dist(x,y)$ between a couple of points $x,y \in M$. The corresponding open metric ball centered at $o \in M$ and of radius $R>0$ is denoted by $B_{R}(o)$. The $\riem$ and $\ric$ symbols are used to denote, respectively, the Riemann and the Ricci curvature tensors of $(M,g)$. Finally, the Laplace-Beltrami operator of $(M,g)$ is denoted by $\Delta = \tr \Hess = \div \nabla$. We stress that we are using the sign convention according to which, on the real line, $\Delta = +\frac{d^{2}}{dx^{2}}$.\smallskip

This paper deals with (sub)solutions of elliptic PDEs involving the Schr\"odinger operator
\[
\CL = \Delta - \l(x)
\]
where $\l(x)$ is a smooth function.

We say that $u \in L^{1}_{loc}(M)$ is a {\it distributional solution} of $\CL u \geq f  \in L^{1}_{loc}(M)$ if, for every $0 \leq \vp \in C^{\infty}_{c}(M)$,
\[
\int_{M}u \CL\vp \geq \int_{M}f \vp.
\]
Sometimes, we will call $u$ a {\it distributional subsolution} of the equation $\CL u = f$. The notion of {\it distributional supersolution} is defined by reversing the inequalities and we say that $u \in L^{1}_{loc}(M)$ is a {\it distributional solution} of $\CL u = f$ if it is a subsolution and a supersolution at the same time.\smallskip

In the presence of more local regularity of the function involved we can also speak of a weak solution of the same inequality. Namely,  $u \in W^{1,1}_{loc}(M)$ is a {\it weak solution} of $\CL u \geq f \in L^{1}_{loc}(M)$ if, for every $0 \leq \vp \in C^{\infty}_{c}(M)$, it holds
\[
- \int_{M}  g(\nabla u , \nabla \vp)    \geq \int_{M} (\l + f) \vp
\]
By a density argument, the inequality can be extended to test functions $0 \leq \vp \in W^{1,\infty}_{c}(M)$. If the regularity of $u$ is increased to $W^{1,2}_{loc}(M)$ then test functions can be taken in $W^{1,2}_{c}(M)$.

Finally, we need to recall that  a function $u \in W^{1,1}_{loc}(M)$ is a distributional solution of $\CL u \geq f \in L^{1}_{loc}$ if and only if it is a weak solution of the same inequality.


\subsection{The BMS conjecture}

According to B. G\"uneysu, \cite{Gu-JGEA}, we set the next

\begin{definition}
Let $1 \leq p \leq +\infty$. The Riemannian manifold $(M,g)$ is said to be $L^{p}$-Positivity Preserving if the following implication holds true:
\begin{equation}\label{P}\tag{$\mathrm{P}_{p}$}
\begin{cases}
(- \Delta + 1) u \geq 0 \text{ distributionally on M}\\
u \in L^{p}(M)
\end{cases}
\Longrightarrow
u \geq 0\text{ a.e. on }M.
\end{equation}
More generally, one can consider any family of functions $\C \subseteq L^{1}_{loc}(M)$ and say that $(M,g)$ is $\C$-Positivity Preserving if the above implication holds when $L^{p}(M)$ is replaced by $\C$.
\end{definition}

The following conjecture, motivated by the study of self-adjointness of covariant Schr\"odinger operators (see the discussions in \cite{Gu-survey, Gu-book}) was formulated by M. Braverman, O. Milatovic and M. Shubin in \cite[Appendix B]{BMS}.

\begin{conjecture}[BMS conjecture]
Assume that $(M,g)$ is geodesically complete.  Then $(M,g)$ is $L^{2}$-positivity preserving.
\end{conjecture}

The validity of the BMS conjecture has been verified under additional restrictions on the geometry of the complete Riemannian manifold $(M,g)$. More precisely:
\begin{itemize}
\item In the seminal paper \cite[p.140]{Ka},  T. Kato proved that $\rr^{n}$ is $L^{2}$-Positive Preserving.
 \item  In \cite[Proposition B.2]{BMS} it is assumed  that $(M,g)$ has $C^{\infty}$-bounded geometry, i.e., it satisfies $\| \nabla^{(j)}\riem \|_{L^{\infty}} < +\infty$ for any $j\in \nn$ and $\inj(M)>0$.
 \item  In \cite{Gu-JGEA}, B. G\"uneysu showed that $\ric \geq 0$  is enough to conclude. Subsequently, in \cite[Theorem XIV.31]{Gu-book}, he proved that if $\ric \geq -K^{2}$ then $(M,g)$ is $L^{p}$-Positivity Preserving on the whole scale $p \in [1,+\infty]$.
 \item  In \cite{BS}, D. Bianchi and A.G. Setti observed that the BMS conjecture is true even if the Ricci curvature condition is relaxed to
 \[
 \ric \geq - C (1+ r(x))^{2}
 \]
where $r(x) = \dist(x,o)$ for some origin $o \in M$. Under the same curvature assumptions, the $L^p$-Positivity Preserving can be extended almost directly to any $p\in[2,\infty)$, and with a little more effort to any $p\in[1,+\infty]$, \cite{Gu-book,MV}.
 \item  In the very recent \cite{MV}, L. Marini and the second author considered the case of a Cartan-Hadamard manifold (complete, simply connected with $\riem \leq 0$). In this setting it is proved that $(M,g)$ is $L^{p}$-Positivity Preserving for any $p \in [2,+\infty)$ provided
 \[
 -(m-1)B^{2}(1+r(x))^{\a +2} \leq \ric \leq -(m-1)^{2}A^{2}(1+ r(x))^{\a}
 \]
 for some $\a>0$ and  $B > \sqrt{2}(m-1)A>0$.
\end{itemize}

Kato's argument in $\rr^n$ relies on the positivity of the operator $(-\Delta + 1)^{-1}$ acting on the space of tempered distributions, which in turn is proved using the explicit expression of its kernel. Instead, in all the above quoted works on Riemannian manifolds, the proofs stem from an argument by B. Davies, \cite[Proposition B.3]{BMS} that relies on the existence of good cut-off functions with controlled gradient and Laplacian. Obviously, the construction of these cut-offs requires some assumption on the curvature.

\subsection{Main result}

In the present  paper we approach the $L^{p}$-Positivity Preserving following a different path which is based on a new \textsl{a priori} regularity result for 
positive subharmonic distributions (see Section \ref{sect:reg}), which permits to prove a Liouville type theorem (see Section \ref{section_Liouville}), and a Brezis-Kato inequality on Riemannian manifolds (see Section \ref{sec:Kato}). Both these results, in turn, rely on a smooth monotonic approximation of distributional solutions of $\CL u \geq 0$ of independent interest (see Section \ref{sect:approx}). This strategy will enable us to avoid any curvature restriction and prove the next

\begin{theorem}\label{th:BMS}
 Let $(M,g)$ be a complete Riemannian manifold. Then $M$ is $L^{p}$-Positivity Preserving for every $p \in (1,+\infty)$. In particular, the BMS conjecture is true.
\end{theorem}

As a metter of fact, a bypass product of our approach is that a complete manifold is $\C$-Positivity Preserving where
\[
\C = \{ u \in L^{1}_{loc}(M): \| u \|_{L^{p}(B_{2R}\setminus B_{R})} = o(R^{2/p}) \text{ as }R \to +\infty \},
\]
where $p\in (1,+\infty)$.
This fact will be pointed out in Remark \ref{rem-subquadratic}.\smallskip

We note explicitly that, in the statement of Theorem \ref{th:BMS}, the endpoint cases $p=1$ and $p=+\infty$ are excluded. If, on the one hand, it is well known that there are complete Riemannian manifolds which are not $L^{\infty}$-Positivity Preserving (since stochastically incomplete), on the other hand we don't know whether or not the conclusion of the Theorem can be extended to $p=1$. See also Remark \ref{rmk-Liouville}.


\section{Smooth  approximation of distributional subsolutions}\label{sect:approx}

In the classical potential theory for the Euclidean Laplacian in $\rr^{n}$ (and therefore on any $2$-dimensional manifold in isothermal local coordinates) it is known that subharmonic distributions are the monotone limit of smooth subharmonic functions. We need to extend this property to the locally uniformly elliptic operator
 \[
 \CL = \Delta - \l(x)
 \]
on a Riemannian manifold, $\l(x)$ being a smooth function. Actually, in view of our purposes, we could even assume that either $\l(x) = 1$ or $\l(x) =0$. We have the following

\begin{theorem}\label{th:smoothing}
	Let $(M,g)$ be an open manifold. For any $0\le \lambda \in C^\infty (M)$, the operator $\CL=\Delta-\lambda(x)$ has the  property of  local smooth monotonic approximation of $L^{1}_{loc}$-subsolutions, i.e the following holds.\\
Every $x_{0}\in M$ has two open neighborhoods $\Omega'\Subset \Omega \Subset M$ such that,  if $u \in L^{1}(\Omega)$ solves $\CL u \geq 0$ in $\Omega$, then there exists  a sequence $\{ u_{k} \} \subseteq C^{\infty}(\overline \Omega')$ satisfying the following properties:
\begin{enumerate}
 \item [a)] $u \leq u_{k+1}  \leq u_{k}$ for all $k \in \nn$;
 \item [b)] $u_{k}(x) \to u(x)$ as $k \to +\infty$ for a.e. $x \in \Omega'$;
 \item [c)] $\CL u_{k} \geq 0$ in $\Omega'$ for all $k \in \nn$;
 \item [d)] $\| u - u_{k} \|_{L^{1}(\Omega')} \to 0$ as $k\to +\infty$.
\end{enumerate}
\end{theorem}

\begin{remark}
	In view of future applications it would be very interesting to weaken the assumptions of Theorem \ref{th:smoothing}, by allowing $\lambda$ to be singular. In that case we expect the $\CL$-subharmonic functions $u_k$ in the approximating sequence to be non-smooth, but still more regular than $u$. 
\end{remark}

\begin{proof}
We preliminarily observe that property d) follows immediately from a) and b). Indeed, since $| u_{k}| \leq \max ( |u_{1}| , |u|)$ we can apply the dominated convergence theorem to conclude. Therefore, we shall concentrate on the construction of a sequence of functions satisfying  properties a), b) and c).\smallskip

Fix $x_0\in M$ and a smooth open neighborhood $B\Subset M$ of $x_0$. Let $\alpha\in C^\infty(M)$ be  a solution of
\[
\begin{cases}
	\CL \alpha = 0,\\
    \alpha >0,
\end{cases}\text{on }B.
\]	
Thanks to the maximum principle, such an $\alpha$ can be obtained for instance as a solution of the Dirichlet problem $\CL \alpha=0$ with positive constant boundary data on $\partial B$. We shall use the following trick introduced by M.H. Protter and H.F. Weinberger in \cite{PW}.
\begin{lemma} The function $w\in L^1_{loc}(M)$ is a distributional solution of $\CL w\ge 0$ on $B$ if and only if $w/\alpha$ is a distributional solution of $\Delta_\alpha (w/\alpha):=\alpha^{-2}\dive(\alpha^2 \nabla(w/\alpha)) \ge 0$ on $B$.
\end{lemma}
\begin{proof}
	Let $0\leq \vp \in C^\infty_c(B)$. Since $\Delta\alpha = \lambda \alpha$, we have
	\begin{align*}
	\alpha \Delta_\alpha(\vp/\alpha) &= \alpha^{-1}\dive (\alpha^2 \nabla(\vp/\alpha))\\
	&= \alpha^{-1} \dive(\alpha\nabla \vp - \vp\nabla \alpha)\\
	&= \alpha^{-1} (\alpha\Delta \vp - \vp\Delta \alpha)\\
    &= \CL \vp.
	\end{align*}	
	Noticing that the operator
$\Delta_{\alpha}$ is symmetric with respect to the smooth weighted measure $\alpha^2\dv$, we get	
	\begin{align*}
		(\Delta_\alpha \frac{w}{\alpha}, \alpha\vp) = \int_{M} \frac{w}{\alpha}\Delta_{\alpha}\frac{\vp}{\alpha} \alpha^2\dv =
		\int_{M} \frac{w}{\alpha}\alpha^2\frac{\CL\vp}{\alpha} \dv =  
		\int_{M} w \CL\vp \dv = (\CL w,\vp). 
	\end{align*}
Since $0\le \alpha\vp \in C^\infty_c(B)$, this concludes the proof of the lemma.
\end{proof}
According to this lemma, setting $v=\alpha^{-1}u$, we can infer the conclusion of the theorem from the equivalent statement for the operator $\Delta_\alpha$. Namely, we \emph{claim} that if $v\in L^1_{loc}(B)$ solves $\Delta_\alpha v \geq 0$ distributionally in $B$, then there exists a smaller neighborhood $B'\Subset B$ of $x_0$ and a sequence $\{ v_{k} \} \subseteq C^{\infty}(\overline B')$ satisfying the following properties:
\begin{enumerate}
	\item [a')] $v \leq v_{k+1}  \leq v_{k}$ for all $k \in \nn$;
	\item [b')] $v_{k}(x) \to v(x)$ as $k \to +\infty$ for a.e. $x \in B'$;
	\item [c')] $\Delta_\alpha v_{k} \geq 0$ in $B'$ for all $k \in \nn$.
\end{enumerate}
Indeed, if we set $\Omega'=B'$, then a'), b') and c') are equivalent to a), b) and c), respectively. 

To prove the claim, choose a chart $\phi:B\to \phi(B)=:\mathbb B\subset \mathbb R^n$ in the differential structure of $M$ and a coordinate system in $\mathbb B$. With respect to this coordinate system $\Delta_\alpha$ writes as
\[
\frac{\alpha^{-2}}{\sqrt{|g|}}\partial_i(\sqrt{|g|} \alpha^2 g^{ij}\partial_j).
\]
Accordingly, the distributional inequality $\Delta_\alpha v \ge 0$ is equivalent to the divergence form inequality
\[
\L v : = \partial_i(a^{ij}\partial_j) v \ge 0\qquad\text{distributionally},
\]
where $[a^{ij}] = [\sqrt{|g|} \alpha^2 g^{ij}]$ is a positive definite symmetric matrix whose entries depend smoothly on $x\in\mathbb B$.  According to \cite[Theorem 1]{Sj}, $v$ is an $\L$-subharmonic function in the sense of the potential theory, cf. \cite{He}. Hence, to conclude, we can apply a slightly modified version of \cite[Theorem 7.1]{BL}. Namely, one can verify that Theorem 7.1 works without any change if one uses in its proof the Green function of $\bb$ with null boundary conditions on $\partial \bb$. The existence of this Green function, in turn, can be deduced using different methods. For instance, from the PDE viewpoint,  we can appeal to the fact that the Dirichlet problem on $\bb$ is uniquely solvable; see the classical \cite{LSW}.

The proof of the theorem is completed.
\end{proof}

\begin{remark}
	The assumption $\lambda \ge 0$ in Theorem \ref{th:smoothing} can be avoided up to taking a small enough neighborhood of $x_0$. Indeed, suppose that $\lambda(x_0)<0$. Since the $M$ is asymptotically Euclidean in $x_0$, the constant in the Poincaré inequality can be made arbitrarily small on small domain. Namely, there exists a small enough neighborhood $U$ of $x_0$ so that 
	\[
 \int_{U} -\lambda \vp^2 \dv \le	-\frac{\lambda(x_0)}{2}\int_{U} \vp^2 \dv\le  \int_U |\nabla \vp|^2 \dv ,\quad \forall \vp\in C^\infty_c(U),\] i.e., the bottom of the spectrum $\lambda^U_1(-\CL)$ of $-\CL=-\Delta+\lambda$ on $U$ is non-negative. By the monotonicity of $\lambda^U_1(-\CL)$ with respect to the domain, we thus obtain that $\lambda^\Omega_1(-\CL)>0$ on some smaller domain $\Omega\Subset U$. Accordingly, there exists a strictly positive solution of $\CL \alpha =0$ on $\Omega$.     
\end{remark}

In the following, we will assume that either
$\l(x) \equiv 0$ or $\l(x) = \l$ is a positive constant. For each of these choices of $\l$, the existence of a smooth local monotone approximation has striking consequences in the regularity theory of  subharmonic distributions or on the validity of a variant of the traditional {\it Kato inequality}.


In the next two sections we are going to analyze separately these applications.

\section{Improved regularity of positive subharmonic distributions}\label{sect:reg}

By a {\it subharmonic distribution} on a domain $\Omega \subseteq M$ we mean a function $u \in L^{1}_{loc}(
\Omega)$ satisfying the inequality $\Delta u \geq 0$ in the sense of distributions. Note that, in this case, $\Delta u$ is a positive Radon measure.

Using the fact that a local monotone approximation exists, positive subharmonic distributions are necessarily in $W^{1,2}_{loc}$. Indeed an even stronger property can be proved, i.e., $u^{p/2}\in W^{1,2}_{loc}$ for all $p\in (1,\infty)$. This is the content of the next Theorem. To the best of our knowledge, in this generality the result is new also in the Euclidean setting.

\begin{theorem}\label{prop:reg positive subharm}
	Let $(M,g)$ be a Riemannian manifold. Let $u \geq 0$ be an $L^1_{loc}(M)$-subharmonic distribution. Then, $u\in L^\infty_{loc}$ and $u^{p/2}\in W^{1,2}_{loc}$ for any $p\in (1,\infty)$. 
\end{theorem}

\begin{proof}
Fix $1 < p < +\infty$ and  let $x_{0} \in M$ be a given point. Let $\Omega\Subset M$ be a neighborhood of $x_{0}$ where it is defined a smooth  approximation of $u$
\[
u_{k} \geq u_{k+1} \geq u \geq 0
\]
by (classical) solutions of $\Delta u_{k} \geq 0$. Note that, up to replacing $u_k$ with $u_k+1/k$, we can suppose that  $u_k>0$. 

Now, let
	$\vp\in C^\infty_c(\Omega)$ to be a cut-off function such that $0\le \vp\le 1$ and $\vp\equiv 1$ on some open set $\Omega_{1}\Subset \Omega$, 
	and define $\psi = u_k^{p-1} \vp^{2}$.
	Since the $u_k$'s are smooth, strictly positive and subharmonic, we obtain
	\begin{align*}
		0 &\geq -\int_M \Delta u_k \psi \\
		& = \int_M g(\nabla u_k,\nabla \psi)\\
		&\geq (p-1)\int_{M} \vp^{2}u_k^{p-2}|\nabla u_k |^{2} + 2\int_{M} \vp u_k^{p-1}g(\nabla u_k,\nabla \vp)\\
		&\ge (p-1-\ve)\int_{M} \vp^{2}u_k^{p-2}|\nabla u_k |^{2} - \ve^{-1} \int_{M} u_k^{p}|\nabla \vp |^{2},
	\end{align*}
	for any $\ve \in (0,p-1)$. Elaborating, we get the Caccioppoli inequality
	\begin{align*}
		\ve(p-1-\ve)\int_{\Omega_{1}} u_k^{p-2}|\nabla u_k|^{2}
		&\leq	\ve(p-1-\ve)\int_{M} \vp^2 u_k^{p-2}|\nabla u_k|^{2}\\ &\leq \int_{M} u_k^{p}|\nabla \vp|^{2}\\
		&\leq \|\nabla \vp\|_\infty^2\int_{\Omega} u_k^{p}.
	\end{align*}
	Using the fact that $u_k\ge u_{k+1}> 0$, this latter implies
	\[
	\int_{\Omega_{1}} |\nabla u_k^{p/2}|^{2}=\frac{p^2}{4}	\int_{\Omega_{1}} u_k^{p-2}|\nabla u_k|^{2}\leq	\frac{p^2\|\nabla \vp\|_\infty^2}{4\ve(p-1-\ve)} \int_{\Omega} u_1^{p},
	\]
	Noticing also that $\|u_k^{p/2}\|_{L^2(\Omega_{1})}\le \|u_1^{p/2}\|_{L^2(\Omega_{1})}$, we have thus obtained that the sequence $\{u_k^{p/2}\}$ is uniformly bounded in $W^{1,2}(\Omega_{1})$, hence weakly converges in $W^{1,2}(\Omega_{1})$ (up to extract a subsequence) to some $v \in W^{1,2}(\Omega_{1})$. Since $u_k^{p/2}$ converges point-wise a.e. to $u^{p/2}$, it holds necessarily $v =u^{p/2}$ a.e. on $\Omega_{1}$.  In particular, $u^{p/2}\in W^{1,2}$ in a neighborhood of $x_{0}$.
\end{proof}


\section{A variant of the Kato inequality}\label{sec:Kato}

The original inequality by T. Kato, \cite[Lemma A]{Ka}, states that if $u \in L^{1}_{loc}(M)$ satisfies $\Delta u \in L^{1}_{loc}(M)$ then
\[
\Delta | u | = \sgn(u) \Delta u
\]
or, equivalently, if we set $u_{+} = \max(u,0) = (u + |u|)/2$, it holds
\[
\Delta u_{+} = 1_{\{ u>0\}} \Delta u.
\]
Note that, in these assumptions, $|\nabla u| \in L^{1}_{loc}(M)$ (see \cite[Lemma 1]{Ka}) and, therefore, $u \in W^{1,1}_{loc}(M)$.

Under the sole requirement that $u \in W^{1,1}_{loc}(M)$ is such that $\Delta u = \mu$ is a (signed) Radon measure, a precise form of the Kato inequality was proved by A. Ancona in \cite[Theorem 5.1]{An}.

In the special case where $M = \rr^{n}$ and $\Delta$ is the Euclidean Laplacian,  H. Brezis, \cite[Lemma A.1]{Br} and  \cite[Proposition 6.9]{Po}, observed that the local regularity of the function can be replaced by the condition that $u$ satisfies a differential inequality of the form $\Delta u \geq f$ in the sense of distributions, where $f \in L^{1}_{loc}$. The proof uses standard mollifiers to approximate $u$ by smooth solutions of the same inequality and, in particular, it works locally on $2$-dimensional Riemannian manifolds thanks to the existence of isothermal  coordinates. For higher dimensional manifolds the proof can be carried out by using the existence of a sequence of smooth approximations. In view of our purposes we limit ourselves to point out the following

\begin{proposition}[Brezis-Kato inequality]\label{prop-BK}
Let $(M,g)$ by any Riemannian manifold. If $u \in L^{1}_{loc}(M)$ satisfies $\CL u  \geq 0$, then $u_{+} \in L^{1}_{loc}(M)$ is a distributional solution of the same inequality, i.e., $\CL u_{+} \geq 0$.
\end{proposition}

\begin{remark}
 In Appendix \ref{section-AnconaKato} we will give an alternative proof of the Brezis-Kato inequality, by deriving it from Ancona's work alluded to above.
\end{remark}

\begin{proof}
 Let $u_{m}$ be a sequence of  monotonic approximation of $u$ by smooth solutions of $\Delta u_{m} \geq \l  u_{m}$ on a domain $V$.
 
Now, let $H:\mathbb R \to \mathbb R$ be a smooth convex function. For any $m\ge 1$,
\begin{align*}
	\Delta (H (u_m)) \ge H'(u_m)\l u_m.
\end{align*}
Therefore, for every $0\le \vp\in C^\infty_c(V)$,
\begin{equation}\label{eq:Hve}
\int_{V} H(u_m)\Delta \vp \dx = \int_{V} \Delta (H(u_m))\vp \dx 
\ge \int_{V} \l u_m H'(u_m)\vp \dx.
\end{equation}
We apply this latter with $H(t)=H_\ve(t)=(t+\sqrt{t^2+\ve})/2$. First, we let $m\to \infty$. 
We get
\[
\left|\int_{V} (H_\ve(u_m)-H_\ve(u)) \Delta \vp \dx\right| \leq \|H_\ve'\|_{L^\infty}\| \Delta\vp\|_{L^\infty}\int_{V} |u_m-u|\dx \longrightarrow 0,
\]
and 
\begin{align*}
&\left|\int_{V} \l u_m H_\ve'(u_m)\vp \dx-\int_{V} \l u H_\ve'(u)\vp \dx\right|\\
\le  &\left|\int_{V} \l (u_m-u) H_\ve'(u_m)\vp \dx\right|+\left|\int_{V} \l u (H_\ve'(u_m)-H_\ve'(u))\vp \dx\right|\\
\le  &\l \|H_\ve'\|_{L^\infty}\|\vp\|_{L^\infty}\int_{V}  |u_m-u|  \dx+\int_{V} \l |u|\, |H_\ve'(u_m)-H_\ve'(u)|\, |\vp| \dx \longrightarrow 0,
\end{align*}
where we used the dominated convergence theorem in the last integral and the $L^1_{loc}$ convergence of the $\{u_m\}$ to $u$ elsewhere. In particular, \eqref{eq:Hve} yields
\[
\int_{V} H_\ve(u) \Delta \vp \dx   
\ge \int_{V} \l uH_\ve'(u)\vp \dx.
\]
Note that $H_\ve(t) \to t_+$ uniformly on $\mathbb R$. Moreover the $H'_\ve(t)$ are uniformly bounded both in $\ve$ and in $t$, and converge pointwise a.e. as $\ve\to 0$ to the characteristic function $\chi_{(0,+\infty)}(t)$. Letting $\ve\to 0$ and applying again the dominated convergence theorem gives
\[
\int_{V} u_+ \Delta\vp \dx   
\ge \int_{V} \l u 1_{\{ u>0 \}}\vp \dx=\int_{V} \l u_+ \vp \dx,
\]
i.e. 
\[
\int_{V} u_+ \Delta\vp \dv   
\ge \int_{V} \l u_+ \vp \dv
\]
for any $0\le\vp\in C^\infty_c(V)$. In order to conclude the proof, take a covering of $M$ by domains $U_{k} \Subset M$ where the monotone approximation exists and consider a subordinated partition of unity $\{\eta_k\}_{k\in K}$ such that $\eta_k\in C^\infty_c(\tilde U_k)$ and $\sum_k\eta_k=1$. Given  $\psi\in C^\infty_c(M)$, one has
\[
\int_{M} u_+ \Delta\psi \dv   
=\sum_k \int_{M} u_+ \Delta(\eta_k\psi) \dv \ge \sum_k \int_{\tilde U_k} \l u_+ \eta_k\psi \dv= \int_{M} \l u_+ \psi \dv.
\]
\end{proof}

\begin{remark}
	We note that only properties c) and d) of the statement of Theorem \ref{th:smoothing} are used in the above proof of the Brezis-Kato inequality. A different approach to obtain such properties could consist in verifying that the approximation theory by R. Greene and H. Wu, \cite{GW}, applies to our operator $\CL$. We are grateful to Batu G\"uneysu for having put this paper to our attention. On the other hand, the monotonicity of the approximating sequence (i.e., property a)) will be vital in the regularity result contained in Section \ref{section_Liouville}, and could reveal useful in future applications.
\end{remark}

A direct application of Theorem \ref{prop:reg positive subharm} and Proposition \ref{prop-BK} gives the following

\begin{corollary}
Let $(M,g)$ be a Riemannian manifold. If $u \in L^{1}_{loc}(M)$ satisfies $\CL u = \Delta u - \l u \geq 0$, with $\l\ge 0$, then $u_{+} \in L^{1}_{loc}(M)$ is a non-negative subharmonic distribution, i.e., $\Delta u_{+} \geq 0$. In particular, $u_{+} \in L^{\infty}_{loc}(M)$ and, for any $p \in (1,+\infty)$, $u_{+}^{p/2} \in W^{1,2}_{loc}(M)$.
\end{corollary}


\section{$L^{p}$-Positivity Preserving from Liouville}\label{section_Liouville}

Once we have a Brezis-Kato inequality, the $L^{p}$-Positivity Preserving  enters in the  realm of Liouville-type theorems for $L^{p}$-subharmonic distributions. This is explained in the next

\begin{lemma} \label{lemma-Lp-Sub}
	Let $(M,g)$ be any Riemannian manifold. Consider the following $L^p$-Liouville property for subharmonic distributions, with $p\in(1,\infty)$:
	\begin{equation}\label{Lp-Sub}\tag{L-Sub$_p$}
		\begin{cases}
			\Delta u  \geq 0  \text{ on }M\\
			u \geq 0 \text{ a.e. on }M\\
			u \in L^{p}(M),
		\end{cases}
		\quad \Longrightarrow \quad
		u \equiv const \text{ a.e. on } M.
	\end{equation}
	Then
	\[
	\eqref{Lp-Sub} \quad \Longrightarrow \eqref{P}.
	\]
\end{lemma}

\begin{proof}
	Suppose $u \in L^{p}(M)$ satisfies $\Delta u \leq u$. By Proposition \ref{prop-BK},
	\[
	\Delta (-u)_{+} \geq  (-u)_{+} \geq 0\text{ on }M,
	\]
i.e. $(-u)_{+} \geq 0$ is a subharmonic distribution. Obviously, $(-u)_{+}\in L^{p}(M)$. It follows from \eqref{Lp-Sub} that $(-u)_{+} = 0$ a.e. on $M$. This means precisely that $u \geq 0$ a.e. on $M$, thus proving the validity of \eqref{P}.
\end{proof}	
  
As we will show in the next theorem, the regularity result proved in  Theorem \ref{prop:reg positive subharm} permits to generalize to our setting a classical Liouville theorem for non-negative $L^p$ subharmonic functions due to S.T. Yau, \cite{Ya}.
\begin{theorem}\label{prop:Lp-Sub}
	Let $p\in(1,\infty)$. Let $(M,g)$ be a complete Riemannian manifold. Then \eqref{Lp-Sub}  holds.
\end{theorem}

\begin{remark}\label{rmk-Liouville}
 The endpoint cases $p=1$ and $p=+\infty$ must be excluded. Indeed, the Hyperbolic space supports infinitely many bounded (hence positive) harmonic functions whereas, on the opposite side, positive, nonconstant, $L^{1}$-harmonic functions on complete Riemann surfaces with (finite volume and) super-quadratic curvature decay to $-\infty$ was constructed by P. Li and R. Schoen in \cite{LS}. Unfortunately, the existence of these functions tells us nothing about the failure of the $L^{1}$-Preservation Property.
\end{remark}

\begin{proof}
For $\d\in(0,1)$, define $u_\d=u+\d>0$ and observe that $u_\delta$ is subharmonic and $u^{q/2}_\d\in W^{1,2}_{loc}(M)\cap L^\infty_{loc}(M)$ for any $q\in (1,\infty)$ by Theorem \ref{prop:reg positive subharm}. 
	Thanks to the local regularity, the subharmonicity condition can be written as
\begin{equation}\label{eq:semiweak subharm}
	- \int_{M}g(\nabla u_\d , \nabla \psi) \geq 0,\quad \forall\,  0\leq \psi \in W^{1,2}_{c}(M).
\end{equation}
Set
\[
\psi = u_\d^{p-1} \vp^{2}
\]
with $0 \leq \vp \in C^{\infty}_{c}(M)$. Since 
\begin{equation}\label{eq:bounds ud}\delta \le |u_\d| \le \|u\|_{L^\infty(\supp \vp)}+\d,\end{equation}
we have that $\psi\in L_c^2(M)$. We need the following lemma which will be proved below.
\begin{lemma}\label{lem:reg psi}
	For all $q>0$, $u_\d^{q}\in W^{1,2}_{loc}(M)$ and the weak gradient of $u_\d^{q}$ satisfies  
	\[
	\nabla u_\d^{q} = qu_\d^{q-1} \nabla u_\d.\]
	\end{lemma} 
\noindent This lemma with $q=p-1$ implies that $\psi \in W_c^{1,2}(M)$ with 
\[\nabla \psi = (p-1)u_\d^{p-2} \vp^{2} \nabla u_\d + u_\d^{p-1} \nabla(\vp^{2}).\]
Thus, we can insert $\psi$ in \eqref{eq:semiweak subharm} and compute	 
	 		\begin{align*}
	 	0 &\geq   (p-1)\int_{M} \vp^{2}u_\d^{p-2}|\nabla u_\d |^{2} + 2\int_{M} \vp u_\d^{p-1}g(\nabla u_\d,\nabla \vp)\\
	 	&\ge (p-1-\ve)\int_{M} \vp^{2}u_\d^{p-2}|\nabla u_\d |^{2} - \ve^{-1} \int_{M} u_\d^{p}|\nabla \vp |^{2},
	 \end{align*}
	 for any $\ve\in (0,p-1)$. Applying Lemma \ref{lem:reg psi} with $q=p/2$, we get the Caccioppoli inequality
	 \begin{align}\label{eq:caccioppoli}
\frac{4\ve(p-1-\ve)}{p^2}\int_{M} \vp^2 |\nabla u_\d^{p/2}|^{2}&=\ve(p-1-\ve)\int_{M} \vp^2 u_\d^{p-2}|\nabla u_\d|^{2}\\
&\leq \int_{M} u_\d^{p}|\nabla \vp|^{2}.\nonumber
	 \end{align}
	Let $\vp=\vp_{k}$ be a sequence of first order cut-off functions with respect to a fixed reference point $o$, i.e. $\vp_k\in C^\infty_c(B_{2k}(o))$ , $\vp_k\equiv 1$ on $B_k(o)$ and $\|\nabla\vp_k\|_{L^\infty}\le 2/k$.
	Inserting $\vp=\vp_k$ in \eqref{eq:caccioppoli}, we obtain that the family $\{u_\d^{p/2}\}$, $\d\in(0,1)$, is uniformly bounded in $W^{1,2}(B_k(o))$. Hence, for some sequence $\d_j\to 0$, the subsequence $u_{\d_j}$ converges weakly in $W^{1,2}(B_k(o))$ to some $v\in W^{1,2}(B_k(o))$. Since $u_{\d_j}^{p/2}$ converges point-wise a.e. to $u^{p/2}$, we deduce that $v=u^{p/2}$. By the weak lower semicontinuity of the energy 
	\begin{equation}\label{eq:caccioppoli-vp}
	\int_{B_k} |\nabla u^{p/2}|^2 \le \liminf_{j} \int_M   u_{\d_j}^{p}|\nabla \vp_k|^{2}=\int_M   u^{p}|\nabla \vp_k|^{2} \le \frac{4}{k^2}\int_{B_{2k}(o)\setminus B_k(o)} u^p,
	\end{equation}
	where we have used the Lebesgue dominated convergence theorem in the middle equality. We finally let $k \to +\infty$, and apply monotone convergence at the LHS and dominated convergence (since $u\in L^p$) at the RHS of \eqref{eq:caccioppoli-vp}, to deduce that $ |\nabla u^{p/2}|^2=0$ a.e. on $M$. In particular $u^{p/2}$ and thus $u$ are constant on $M$. 
\end{proof}

\begin{remark} \label{rem-subquadratic} As it is clear from the proof of Theorem \ref{prop:Lp-Sub}, on any given complete Riemannian manifold and for any $p\in(1,\infty)$, the Liouville property \eqref{Lp-Sub} holds in the stonger form:
\begin{equation*}
	\begin{cases}
		\Delta u  \geq 0  \text{ on }M\\
		u \geq 0 \text{ a.e. on }M\\
		u \in L_{loc}^{p}(M)\text{ and }\|u\|^p_{L^p(B_{2k}(o)\setminus B_k(o))}=o(k^2),
	\end{cases}
	\quad \Longrightarrow \quad
	u \equiv c \text{ a.e. on } M.
\end{equation*}
Accordingly, also the Positivity Preserving property holds in this class of functions larger than $L^p(M)$.
\end{remark}

It remains to prove the lemma.
\begin{proof}[Proof of Lemma \ref{lem:reg psi}]
For any $x_0\in M$, choose a chart $\phi:U\to \phi(U)=:V\subset \mathbb R^n$ in the differential structure of $M$ so that $x_0\in U$, and let  $\{x_1,\dots,x_n\}$ be a local coordinate system defined in $V$. With an abuse of notation, we denote here $u_\d\circ \phi^{-1}$ again by $u_\d$. We know that $u_\d\in W^{1,1}(V)\subset W^{2,2}(V)$.
We \textsl{claim} that 
\[\partial_i(u_\d^{q})=qu_\d^{q-1}\partial_i u_\d,\quad i=1,\dots,n.\] Then trivially
\[\nabla_{\partial_i} (u_\d^{q}) =g^{ji}\partial_j(u_\d^{q})=qu_\d^{q-1}g^{ji}\partial_j u_\d =qu_\d^{q-1}\nabla_{\partial_i} u_\d,\quad i=1,\dots,n,\]
which conclude the proof of the lemma. To prove the claim, we follow the proof of \cite[Lemma 7.5]{GT}.
Let $\{u_{m}\}$ be a sequence of functions in $W^{1,1}(V)$ such that $u_{m}\to u_\d$ strongly in $W^{1,1}_{loc}(V_1)$ for some open set $V_1\Subset V$. The $u_m$'s can be constructed by convolving $u_\d$ with a sequence $\{\rho_m\}$ of
non-negative mollifiers satisfying $\|\rho_m\|_{L^1(V)}=1$ for all $m$; see \cite[Lemmas 7.2 and 7.3]{GT}. Accordingly, by \eqref{eq:bounds ud} we know that
\begin{equation*}
	0<\delta \le |u_m| \le \|u\|_{L^\infty(\supp \vp)}+\d:=\delta',\quad\forall m\ge 1.
\end{equation*}
Set $f(t)=t^{q}$, so that $u_\d^{q}=f(u_\d)$. We note that $|f'|$ is bounded on $[\d,\d']$. This control on the derivative of $f$ is sufficient to reproduce the remaining steps of the proof of \cite[Lemma 7.5]{GT} and thus to conclude the proof of Lemma \ref{lem:reg psi}.
\end{proof}

Combining Lemma \ref{lemma-Lp-Sub} with Theorem \ref{prop:Lp-Sub} proves Theorem \ref{th:BMS}.

\appendix

\section{A different proof of the Brezis-Kato inequality}\label{section-AnconaKato}

We are going to provide a different proof of Proposition \ref{prop-BK} by using another trick to get rid of the linear term of the Schr\"odinger operator $\CL = \Delta - \l(x)$. This trick is inspired to \cite[Section 3]{GW}. 

\begin{proposition}[Brezis-Kato inequality]
Let $(M,g)$ be a Riemannian manifold. If $u \in L^{1}_{loc}(M)$ satisfies $\Delta u \geq \l u$ then $u_{+} \in L^{1}_{loc}(M)$ is a  solution of  $\Delta u_{+} \geq \l u_{+}$.
\end{proposition}

\begin{proof}
 We fix a smooth domain $\Omega \Subset M$ where subharmonic distributions posses a monotone approximation by smooth subharmonic functions. Next, we consider the Dirichlet problem
 \[
\begin{cases}
 \Delta g = \l u & \text{ in }\Omega \\
 g = 0 & \text{ on }\partial \Omega
\end{cases}
 \]
and we note that, since $\l u \in L^{1}(\Omega)$, then it has a unique solution $g\in W^{1,1}_{0}(\Omega)$; see \cite[Theorem 5.1]{LSW}. The new function
 \[
 w = u - g \in L^{1}(\Omega)
 \]
is a subharmonic distribution, namely, $\Delta w \geq 0$.
 
Let  $w_{k} \geq w_{k+1} \geq w$ be a monotonic approximation of $w$ by smooth solutions of $\Delta w_{k} \geq 0$ and define
 \[
 u_{k} = w_{k} + g \in W^{1,1}(\Omega).
 \]
 Then, by construction,
 \[
 i)\, u_{k} \searrow u \text{ a.e. in }\Omega, \quad ii)\, \Delta u_{k} \geq \l u \text{ in }\Omega.
 \]
 We now apply to each function $u_{k} \in W^{1,1}(\Omega)$ the Kato inequality by Ancona, \cite[Theorem 5.1]{An}, and deduce that
 \[
 \Delta (u_{k})_{+} \geq 1_{\{ u_{k}>0\}} \l u \text{ in }\Omega.
 \]
 This means that, for any $0 \leq \vp \in C^{\infty}_{c}(\Omega)$,
 \[
 \int_{\Omega} (u_{k})_{+} \Delta \vp \geq \int_{\Omega}  1_{\{ u_{k}>0\}} \l u  \vp.
 \]
 We elaborate the RHS of this latter as follows:
 \begin{align*}
 \int_{\Omega}  1_{\{ u_{k}>0\}} \l  u  \vp &= \int_{\Omega}  1_{\{ u_{k}>0\}} \l (u-u_{k}) \vp
 + \int_{\Omega}  1_{\{ u_{k}>0\}} \l u_{k}  \vp \\
 &= \int_{\Omega}  1_{\{ u_{k}>0\}} \l (u-u_{k})  \vp + \int_{\Omega}  \l (u_{k})_{+} \vp \\
 & \to \int_{\Omega} \l u_{+} \vp, \quad \text{ as }k \to +\infty,
 \end{align*}
where, in the last line, we have used the dominated convergence theorem. Since, on the other hand,
\[
\int_{\Omega} (u_{k})_{+} \Delta \vp \to \int_{\Omega} u_{+} \Delta \vp \quad \text{ as }k\to +\infty
\]
we conclude that, for all $0 \leq \vp \in C^{\infty}_{c}(\Omega)$,
\[
\int_{\Omega} u_{+} \Delta \vp  \geq \int_{\Omega} \l u_{+} \vp.
\]
This means that
\[
\Delta u_{+} \geq \l u_{+} \quad \text{ in }\Omega.
\]
Using a partition of unity argument, the last differential inequality extends to any relatively compact domain.
\end{proof}

\subsection*{Acknowledgements.} We are indebted to Andrea Bonfiglioli and Ermanno Lanconelli for clarifying some issues about their article \cite{BL}, which has revealed fundamental in our Section \ref{sect:approx}. We also thanks Peter Sj\"ogren for sharing with us some useful comments about the Brelot-Hervé theory of subharmonic functions and the smoothing approximation procedure. Finally, we are very grateful to Batu G\"uneysu for pointing out to us the BMS conjecture, and for several fruitful discussions on the topic during the last years.


\end{document}